\documentclass[a4paper,11pt]{amsart}
\usepackage[left=25mm, right=25mm, top=15mm, bottom=15mm, includeheadfoot]{geometry}	

\usepackage[T1]{fontenc}

\usepackage[leqno]{amsmath}
\usepackage{amsthm}
\usepackage{amsfonts}
\usepackage{amssymb}

\usepackage[utf8]{inputenc}
\usepackage[english]{babel}

\usepackage{tgschola}
\usepackage{tgadventor} 
\usepackage{tgpagella} 

\usepackage{enumitem}

\usepackage{color}

\theoremstyle{plain}
\newtheorem{lem}{Lemma}[section]
\newtheorem{thm}{Theorem}
\newtheorem{prop}[lem]{Proposition}

\theoremstyle{remark}
\newtheorem{rem}[lem]{Remark}

\theoremstyle{definition}
\newtheorem{defi}[lem]{Definition}

\providecommand{\abs}[1]{\lvert#1\rvert} 
\providecommand{\norm}[1]{\lVert#1\rVert}

\DeclareMathOperator{\ud}{\!d\!}
\DeclareMathOperator{\Ud}{D}
\DeclareMathOperator{\dist}{dist}

\DeclareMathOperator{\esssup}{ess\, sup}
\DeclareMathOperator{\Div}{div}
\DeclareMathOperator{\Rot}{rot}
\DeclareMathOperator{\Dt}{\frac{\ud}{\ud t}}

\DeclareMathOperator{\Id}{Id}

\numberwithin{equation}{section}

\begin{document}

\title{Long-time behavior of micropolar fluid equations in cylindrical domains}
\author{Bernard Nowakowski}
\thanks{The author is partially supported by Polish KBN grant N N201 393137}
\address{Bernard Nowakowski\\ Institute of Mathematics\\ Polish Academy of Sciences\\ \'Snia\-deckich 8\\ 00-956 Warsaw\\ Poland}
\email{bernard@impan.pl}


\date{November, 2011}

\begin{abstract}
	In this paper we investigate the existence of $H^1$-uniform attractor and long-time behavior of solutions to non-autonomous micropolar fluid equations in three dimensional cylindrical domains. In our considerations we take into account that existence of global and strong solutions is proved under the assumption on smallness of change of the initial and the external data along the axis of the cylinder. Therefore, we refine the concept of uniform attractor by adopting the idea which was proposed by J.W. Cholewa and T. Dłotko in \cite{chol1}.
\end{abstract}

\maketitle

\section{Introduction}

The study of attractors is an important part of examining dynamical systems. It was thoroughly investigated in many works (see e.g. \cite{Hale:1988fk}, \cite{Ladyzhenskaya:1991uq}, \cite{Temam:1997vn}, \cite{chol}, \cite{che}, \cite{Sell:2002kx}). These classical results cover many both autonomous and non-autonomous equations of mathematical physics. Roughly speaking, the general abstract theory justifying attractors' existence may be applied when the equations describing autonomous systems posses unique, strong and global solutions. For non-autonomous systems some additional uniformity for external data is required. Further obstacles arise when unbounded domains are examined (see e.g. \cite{luk5} and \cite{zhao} and references therein).

In three dimensions the situation for nonlinear evolution problems is even more complex. In case of micropolar equations we lack information about the regularity of weak solutions at large. Thus, results concerning the existence of uniform attractors are merely conditional (see e.g. \cite{che}). In this article we present some remedy which is based on \cite{chol1}. It takes into account that global and strong solutions exist if some smallness of $L_2$-norms on the rate of change of the external and the initial data is assumed (see \cite{2012arXiv1205.4507N}). This leads to a restriction of the uniform attractor to a proper phase space.

Introduced by A.~Eringen in \cite{erin}, the micropolar fluid equations form a useful generalization of the classical Navier-Stokes model in the sense that they take into account the structure of the media they describe. With many potential applications (see e.g. \cite{pop1}, \cite{pop2}, \cite{ari}, \cite{luk1}) they became an interesting and demanding area of interest for mathematicians and engineers (see e.g. \cite{erin}, \cite{pad}, \cite{gald}, \cite{ort}, \cite{luk1}, \cite{luk2} et al.)

In this article we study the following initial-boundary value problem
\begin{equation}
	\begin{aligned}\label{p1}
		&v_{,t} + v\cdot \nabla v - (\nu + \nu_r)\triangle v + \nabla p = 2\nu_r\Rot\omega + f   & &\text{in } \Omega^{\infty} := \Omega\times(0,\infty),\\
		&\Div v = 0 & &\text{in } \Omega^{\infty},\\
		&\begin{aligned}
			&\omega_{,t} + v\cdot \nabla \omega - \alpha \triangle \omega - \beta\nabla\Div\omega + 4\nu_r\omega \\
			&\mspace{60mu} = 2\nu_r\Rot v + g
		\end{aligned}& &\text{in } \Omega^{\infty}, \\
		&v\cdot n = 0 & &\text{on $S^{\infty} := S\times(0,\infty)$}, \\
		&\Rot v \times n = 0  & &\text{on $S^{\infty}$}, \\
		&\omega = 0 & &\text{on $S_1^{\infty}$}, \\
		&\omega' = 0, \qquad \omega_{3,x_3} = 0 & &\text{on $S_2^{\infty}$}, \\	
		&v\vert_{t = 0} = v(0), \qquad \omega\vert_{t = 0} = \omega(0) & &\text{in $\Omega$}.
    \end{aligned}
\end{equation}
By $\Omega \subset \mathbb{R}^3$ we mean a cylindrical type domain, i.e.
\begin{equation*}
	\Omega = \left\{\left(x_1,x_2\right)\in\mathbb{R}^2\colon \varphi\left(x_1,x_2\right) \leq c_0\right\} \times \left\{ x_3\colon -a \leq x_3 \leq a\right\},
\end{equation*}
where $a > 0$ and $c_0$ are real constants, $\varphi$ is a closed curve of class $\mathcal{C}^2$. The functions $v\colon\mathbb{R}^3\to\mathbb{R}^3$, $\omega\colon\mathbb{R}^3\to\mathbb{R}^3$ and $p\colon\mathbb{R}^3\to\mathbb{R}$ denote the velocity field, the micropolar field and the pressure. The external data (forces and momenta) are represented by $f$ and $g$ respectively. The viscosity coefficients $\nu$ and $\nu_r$ are real and positive. 

The choice of the boundary conditions for $v$ and $\omega$ was thoroughly described in \cite{2012arXiv1205.4046N}. Note that usually the zero Dirichlet condition is assumed for both the velocity and micropolar field. From physical point of view this is not suitable for some classes of fluids (see e.g. \cite{bou}). In fact, the explanation of particles behavior at the boundary is far more complex and was extensively discussed in \cite{con}, \cite{all3}, \cite{mig} and \cite{eri6}.

The article is organized as follows: in the next Section we present an overview of the current state of the art. Subsequently we introduce notation and give some preliminary results. In Section \ref{sec4} the main results are formulated. Section \ref{sec5} is entirely devoted to the basics of semi-processes. In Sections \ref{sec6}, \ref{sec7} and \ref{sec8} we prove the existence of the restricted uniform attractor and analyze the behavior of the micropolar fluid flow for the large viscosity $\nu$. We demonstrate that for time independent data the global and unique solution $(v(t),\omega(t))$ converges to the solutions of the stationary problem and if $\nu_r \approx 0$ then the trajectories $(v(t),\omega(t))$ of micropolar flow and $u(t)$ of the Navier-Stokes flow differ indistinctively.

\section{State-of-the-art}

The study of the uniform attractors has mainly been focused on the the case of unbounded domains (see \cite{ukaszewicz:2004kx} and \cite{Dong:2006uq} and the references therein) which is barely covered by general abstract theory introduced in \cite{che}. Note that both these studies cover only two dimensional case. In three dimensions no results have been obtained so far due to lack of regularity of weak solutions. Another obstacle arises from relaxing the assumption of certain uniformity of the external data. To avoid these difficulties, the concepts of trajectory and pull-back attractors have been suggested and applied (see \cite{Chen:2007zr}, \cite{Chen:2009ys}, \cite{luk7}, \cite{Kapustyan:2010vn} and \cite{tar2}).

\section{Notation and preliminary results}\label{sec3}

In this paper we use the following function spaces:
\begin{enumerate}
	\item[$\bullet$] $W^m_p(\Omega)$, where $m \in \mathbb{N}$, $p \geq 1$, is the closure of $\mathcal{C}^{\infty}(\Omega)$ in the norm
		\begin{equation*}
			\norm{u}_{W^m_p(\Omega)} = \left(\sum_{\abs{\alpha} \leq m} \norm{\Ud^{\alpha} u}_{L_p(\Omega)}^p\right)^{\frac{1}{p}},
		\end{equation*}
	\item[$\bullet$] $H^k(\Omega)$, where $k \in \mathbb{N}$, is simply $W^k_2(\Omega)$, 
	\item[$\bullet$] $W^{2,1}_p(\Omega^t)$, where $p \geq 1$, is the closure of $\mathcal{C}^{\infty}(\Omega\times(t_0,t_1))$ in the norm
		\begin{equation*}
			\norm{u}_{W^{2,1}_p(\Omega^t)} = \left(\int_{t_0}^{t_1}\!\!\!\int_{\Omega} \abs{u_{,xx}(x,s)}^p + \abs{u_{,x}(x,s)}^p + \abs{u(x,s)}^p + \abs{u_{,t}(x,s)}^p\, \ud x\, \ud s\right)^{\frac{1}{p}},
		\end{equation*}
	\item[$\bullet$] $\mathcal{V}$ is the set of all smooth solenoidal functions whose normal component vanishes on the boundary
		\begin{equation*}
			\mathcal{V} = \left\{u\in\mathcal{C}^\infty(\Omega)\colon \Div u = 0 \text{ in } \Omega, \ u\cdot n\vert_S = 0 \right\},
		\end{equation*}
	\item[$\bullet$] $H$ is the closure of $\mathcal{V}$ in $L_2(\Omega)$, 
	\item[$\bullet$] $\widehat H = H \times L_2(\Omega)$,
	\item[$\bullet$] $V$ is the closure of $\mathcal{V}$ in $H^1(\Omega)$,
	\item[$\bullet$] $\widehat V = V\times H^1(\Omega)$, 
	\item[$\bullet$] $V^k_2(\Omega^t)$, where $k \in \mathbb{N}$, is the closure of $\mathcal{C}^{\infty}(\Omega\times(t_0,t_1))$ in the norm
		\begin{equation*}
			\norm{u}_{V^k_2(\Omega^t)} = \underset{t\in (t_0,t_1)}{\esssup}\norm{u}_{H^k(\Omega)} \\
    +\left(\int_{t_0}^{t_1}\norm{\nabla u}^2_{H^{k}(\Omega)}\, \ud t\right)^{1/2}.
		\end{equation*}
\end{enumerate}

To shorten the energy estimates we use:
\begin{equation}\label{eq14}
	\begin{aligned}
		E_{v,\omega}(t) &:= \norm{f}_{L_2(t_0,t;L_{\frac{6}{5}}(\Omega))} + \norm{g}_{L_2(t_0,t;L_{\frac{6}{5}}(\Omega))} + \norm{v(t_0)}_{L_2(\Omega)} + \norm{\omega(t_0)}_{L_2(\Omega)}, \\
		E_{h,\theta}(t) &:= \norm{f_{,x_3}}_{L_2(t_0,t;L_{\frac{6}{5}}(\Omega))} + \norm{g_{,x_3}}_{L_2(t_0,t;L_{\frac{6}{5}}(\Omega))} + \norm{h(t_0)}_{L_2(\Omega)} + \norm{\theta(t_0)}_{L_2(\Omega)}.
	\end{aligned}
\end{equation}
The following function is of particular interest
\begin{equation}\label{eq26}
	\delta(t) := \norm{f_{,x_3}}^2_{L_2(\Omega^t)} + \norm{g_{,x_3}}^2_{L_2(\Omega^t)} + \norm{\Rot h(t_0)}^2_{L_2(\Omega)} + \norm{h(t_0)}^2_{L_2(\Omega)} + \norm{\theta(t_0)}^2_{L_2(\Omega)}.
\end{equation}
It expresses the smallness assumption which has to be made in order to prove the existence of regular solutions (see \cite{2012arXiv1205.4046N}. It contains no $L_2$-norms of the initial or the external data but only $L_2$-norms of their derivatives alongside the axis of the cylinder. In other words the data need not to be small but small must be their rate of change.

With the above notation the following Theorem was proved in \cite{2012arXiv1205.4507N}. It is fundamental for further considerations.
\begin{thm}\label{thm0}
	Let $0 < T < \infty$ be sufficiently large and fixed. Suppose that $v(0), \omega(0) \in H^1(\Omega)$ and $\Rot h(0) \in L_2(\Omega)$. In addition, let the external data satisfy $f_3\vert_{S_2} = 0$, $g'\vert_{S_2} = 0$, 
	\begin{equation*}	
		\begin{aligned}
			&\norm{f(t)}_{L_2(\Omega)} \leq \norm{f(kT)}_{L_2(\Omega)}e^{-(t - kT)}, & & &\norm{f_{,x_3}(t)}_{L_2(\Omega)} &\leq \norm{f_{,x_3}(kT)}_{L_2(\Omega)}e^{-(t - kT)},\\
			&\norm{g(t)}_{L_2(\Omega)} \leq \norm{g(kT)}_{L_2(\Omega)}e^{-(t - kT)}, & & &\norm{g_{,x_3}(t)}_{L_2(\Omega)} &\leq \norm{g_{,x_3}(kT)}_{L_2(\Omega)}e^{-(t - kT)}
		\end{aligned}
	\end{equation*}
	and 
	\begin{equation*}
		\sup_k \left\{f(kT),f_{,x_3}(kT),g(kT),g_{,x_3}(kT)\right\} < \infty.
	\end{equation*}
	Then, for sufficiently small $\delta(T)$ there exists a unique and regular solution to problem \eqref{p1} on the interval $(0,\infty)$. Moreover, 
	\begin{multline*}
		\norm{v}_{W^{2,1}_2(\Omega^{\infty})} + \norm{\omega}_{W^{2,1}_2(\Omega^{\infty})} + \norm{\nabla p}_{L_2(\Omega^{\infty})} \leq \sup_k \Big(\norm{f}_{L_2(\Omega^{kT})} + \norm{f_{,x_3}}_{L_2(\Omega^{kT})} \\
		+ \norm{g}_{L_2(\Omega^{kT})} + \norm{g_{,x_3}}_{L_2(\Omega^{kT})} + \norm{v(0)}_{H^1(\Omega)} + \norm{\omega(0)}_{H^1(\Omega)} + 1\Big)^3.
	\end{multline*}
\end{thm}

\section{Main results}\label{sec4}

The main result of this paper reads:
\begin{thm}\label{thm1}
	Let the assumption from Theorem \ref{thm0} be satisfied. Suppose that
	\begin{equation*}
		\norm{u\big((k + 1)T + t\big) - u(kT + t)}_{L_2(\Omega)} < \epsilon, 
	\end{equation*}
	for any $t \in [0,T]$ and $k \in \mathbb{N}$, where $u$ is any element of the set $\{f,g,f_{,x_3},g_{,x_3}\}$. Then, the family of semi-processes $\{U_\sigma(t,\tau)\colon t\geq\tau\geq 0\}$, $\sigma \in \Sigma^{\epsilon}_{\sigma_0}$ corresponding to problem \eqref{p1} has an uniform attractor $\mathcal{A}_{\Sigma^{\epsilon}_{\sigma_0}}$ which coincides with the uniform attractor $\mathcal{A}_{\omega(\Sigma^{\epsilon}_{\sigma_0})}$ of the family of semiprocesses $\{U_\sigma(t,\tau)\colon t\geq\tau\geq 0\}$, $\sigma \in \omega\left(\Sigma^{\epsilon}_{\sigma_0}\right)$, i.e.
	\begin{equation*}
		\mathcal{A}_{\Sigma^{\epsilon}_{\sigma_0}} = \mathcal{A}_{\omega\big(\Sigma^{\epsilon}_{\sigma_0}\big)}.
	\end{equation*}
\end{thm}
For explanatory notation we refer to Section \ref{sec5}.

The second result deals with problem \eqref{p1} when the external data are time independent and the viscosity $\nu$ is large enough. It is expected that in such a case the solutions tend to stationary solutions.

\begin{thm}[convergence to stationary solutions]\label{thm3}
	Suppose that the external data $f$ and $g$ do not depend on time. Then, there exists a positive constant $\nu_*$ such that for all $\nu > \nu_*$ the stationary problem 
	\begin{equation*}
		\begin{aligned}
			&v_{\infty}\cdot \nabla v_{\infty} - (\nu + \nu_r)\triangle v_{\infty} + \nabla p_{\infty} = 2\nu_r\Rot\omega_{\infty} + f   & &\text{in } \Omega,\\
			&\Div v_{\infty} = 0 & &\text{in } \Omega,\\
			&\begin{aligned}
				&v_{\infty}\cdot \nabla \omega_{\infty} - \alpha \triangle \omega_{\infty} - \beta\nabla\Div\omega_{\infty} + 4\nu_r\omega_{\infty} \\
				&\mspace{60mu} = 2\nu_r\Rot v_{\infty} + g
			\end{aligned}& &\text{in } \Omega, \\
			&v_{\infty}\cdot n = 0 & &\text{on $S$}, \\
			&\Rot v_{\infty} \times n = 0  & &\text{on $S$}, \\
			&\omega_{\infty} = 0 & &\text{on $S_1$}, \\
			&\omega'_{\infty} = 0, \qquad \omega_{\infty 3,x_3} = 0 & &\text{on $S_2$}
    		\end{aligned}
	\end{equation*}
	has a unique solution $(v_{\infty},\omega_{\infty})$ for which
	\begin{equation}\label{eq43}
		\norm{v_\infty}_{H^2(\Omega)} + \norm{p_{\infty}}_{H^1(\Omega)} + \norm{\omega_{\infty}}_{H^2(\Omega)} \leq F\big(\norm{f}_{L_2(\Omega)},\norm{g}_{L_2(\Omega)}\big)
	\end{equation}
	holds, where $F\colon[0,\infty)\times[0,\infty]\to[0,\infty)$, $F(0,0) = 0$ is a continuous and increasing function. Moreover, under assumptions of Theorem \ref{thm1} there exists a unique solution $(v(t),\omega(t))$ to problem \eqref{p1} converging to the the stationary solution $(v_{\infty},\omega_{\infty})$ as $t\to\infty$ and satysfying
	\begin{equation*}
		\norm{v(t) - v_{\infty}}^2_{L_2(\Omega)} + \norm{\omega(t) - \omega_{\infty}}^2_{L_2(\Omega)} \leq \left(\norm{v(0) - v_{\infty}}^2_{L_2(\Omega)} + \norm{\omega(0) - \omega_{\infty}}^2_{L_2(\Omega)}\right)e^{-\Delta(\nu)t}, 
	\end{equation*}
	where
	\begin{equation*}
		\Delta(\nu) = c_1(\nu) - \frac{3c_2}{\nu}F^4\big(\norm{f_\infty}_{L_2(\Omega)},\norm{g_\infty}_{L_2(\Omega)}\big)
	\end{equation*}
	and
	\begin{equation*}
		\begin{aligned}
			c_1(\nu) &= \frac{\min\{\nu,\alpha\}}{c_{\Omega}}, \\
			c_2 &= \frac{c_{I,\Omega}}{\min\{1,\alpha\}}.
		\end{aligned}
	\end{equation*}
\end{thm}

The last result of this articles establishes certain relation between the trajectories of the standard Navier-Stokes equations and micropolar equations. 

\begin{thm}\label{thm2}
	Let the pair $(u,\Theta)$ be a solution to the following initial-boundary value problem 
	\begin{equation}\label{p11}
		\begin{aligned}
			&u_{,t} - \nu\triangle u + (u\cdot \nabla) u + \nabla q = f & &\text{in $\Omega^t$}, \\
			&\Div u = 0 & &\text{in $\Omega^t$}, \\
			&\Theta_{,t} - \alpha \triangle \Theta - \beta \nabla \Div\Theta + (u\cdot \nabla)\Theta = g & &\text{in $\Omega^t$}, \\
			&\Rot u \times n = 0 & &\text{on $S^t$}, \\
			&u \cdot n = 0 & &\text{on $S^t$}, \\
			&\Theta = 0 & &\text{on $S_1^t$}, \\
			&\Theta' = 0, \qquad \Theta_{3,x_3} = 0 & &\text{on $S_2^t$}, \\
			&u\vert_{t = t_0} = u(t_0), \qquad \Theta\vert_{t = t_0} = \Theta(t_0) & &\text{in $\Omega\times\{t = t_0\}$}.
		\end{aligned}
	\end{equation}
	Suppose that $\nu > 0$ is sufficiently large. Finally, let the assumptions of Theorem \ref{thm1} hold. Then, for any $(u(t_0),\Theta(t_0)) \in B_{(v(t_0),\omega(t_0))}(R)$ (i.e. ball centered at $(v(t_0),\omega(t_0))$ with radius $R$), where $R > 0$, there exists $t^* = t^*(R)$ such that for all $t \geq t^*$ the trajectory $(v(t),\omega(t))$ lies in the $\epsilon$-neighborhood of the trajectory $(u(t),\Theta(t))$.
\end{thm}

\section{Basics of semiprocesses}\label{sec5}

We begin with recalling a few facts from \cite[Ch. 1]{chol} and \cite[Ch. 2]{che}.

Let $\{T(t)\}$ be a semigroup acting on a complete metric or Banach space $X$. Denote by $\mathcal{B}(X)$ the set of all bounded sets in $X$ with respect to metric in $X$. We say that $P \subset X$ is an \emph{attracting set} for $\{T(t)\}$ if for any $B \in \mathcal{B}(X)$
\begin{equation*}
	\dist_X\big(T(t)B,P\big)\xrightarrow[t\to\infty]{} 0.
\end{equation*}

Now we may define an attractor:
\begin{defi}
	A set $\mathcal{A}$ is called \emph{global attractor} for the semigroup $\{T(t)\}$, if it satisfies:
	\begin{itemize}
		\item $\mathcal{A}$ is compact in $X$,
		\item $\mathcal{A}$ is an attracting set for $\{T(t)\}$,
		\item $\mathcal{A}$ is strictly invariant with respect to $\{T(t)\}$, i.e. $T(t)\mathcal{A} = \mathcal{A}$ for all $t \geq 0$.
	\end{itemize}
\end{defi}

\begin{defi}
	For any $B \in \mathcal{B}(X)$ the set
	\begin{equation*}
		\omega(B) = \bigcap_{\tau \geq 0} \overline{\bigcup_{t\geq \tau}T(t) B}^{X}
	\end{equation*}
	is called an \emph{$\omega$-limit set} for $B$.
\end{defi}

The existence of the global attractor in ensured by the following Proposition:
\begin{prop}\label{prop2}
	Suppose that $\{T(t)\}$ is a continuous semigroup in a complete metric space $X$ having a compact attracting set $K \Subset X$. Then the semigroup $\{T(t)\}$ has a global attractor $\mathcal{A}$ ($\mathcal{A} \Subset K$). The attractor coincides with $\omega(K)$, i.e. $\mathcal{A} = \omega(K)$. 
\end{prop}

\begin{proof}
	We refer the reader to \cite[Ch. 2, \S 3, Theorem 3.1]{che}.
\end{proof}

To give the proof of Theorem \ref{thm1} in the next Section we introduce the notions and concept of semi-processes. For this purpose we make use of \cite[Ch. $7$]{che}.

Let us rewrite equation \eqref{p1}$_1$ in the abstract form
\begin{equation}\label{eq37}
	(v_{,t},\omega_{,t}) = A(v,\omega,t) = A_{\sigma(t)}(v,\omega), \qquad t\in \mathbb{R}^+,
\end{equation}
where the right-hand side depends directly on the time-dependent forces and momenta, which is indicated by the presence of function $\sigma(t) = (f(t),g(t))$. The function $\sigma(t)$ is referred as the \emph{time symbol} (or the \emph{symbol}).

By $\Psi$ we denote a Banach space, which contains the values of $\sigma(t)$ for almost all $t \in \mathbb{R}_+$. In our case $\Psi = L_{\frac{6}{5}}(\Omega)\cap L_2(\Omega)\times L_{\frac{6}{5}}(\Omega)\cap L_2(\Omega)$ (although we could write $L_2(\Omega)\times L_2(\Omega)$ only since $\Omega$ is bounded but we want to point out that in certain cases a weaker assumption on forces can be imposed). Moreover we assume that $\sigma(t)$, as a function of $t$, belongs to a topological space $\Xi_+ := \{\xi(t), t \in\mathbb{R}_+\colon \xi(t) \in \Psi \ \text{for a.e. $t\in\mathbb{R}_+$}\}$. It is tempting to write $\Xi_+ = L_{2,loc}(0,\infty;\Psi)$ but we must take into account certain restrictions, which were imposed on the data (see Theorem \ref{thm1}). Thus, we describe $\Xi_+$ in greater detail below.

Consider then the family of equations of the form of \eqref{eq37}, where $\sigma(t) \in \Sigma \subseteq \Xi_+$. The space $\Sigma$ is called \emph{symbol space} and is assumed to contain, along with $\sigma(t)$, all translations, i.e. $\sigma(t + s) = T(s)\sigma(t) \in \Sigma$ for all $s\geq 0$, where $T(s)$ is a translation operator acting on $\Xi_+$. Furthermore, we fix $\sigma_0(t)$, $t \geq 0$. Consider the closure in $\Xi_+$ of the set
\begin{equation*}
	\left\{T(s)\sigma_0(t)\colon s \geq 0\right\} = \left\{\sigma_0(t + s)\colon s\geq 0\right\}.
\end{equation*}
We call this closure the \emph{hull} of the symbol $\sigma_0(t)$ and denote by $\mathcal{H}_+(\sigma_0)$,
\begin{equation*}
	\mathcal{H}_+(\sigma_0) = \overline{\left\{T(s)\sigma_0\colon s \geq 0\right\}}^{\Xi_+}.
\end{equation*}
\begin{defi}
	The symbol $\sigma_0(t) \in \Xi_+$ is called \emph{translation compact} in $\Xi_+$ if the hull $\mathcal{H}_+(\sigma_0)$ is compact in $\Xi_+$.
\end{defi}

In view of the above definition we would like to set $\Sigma = \mathcal{H}_+(\sigma_0)$. However for this purpose we need to determine how to choose $\Sigma$ to ensure its compactness in $\Xi_+$. This is crucial for the proof of the existence of the uniform attractor because we use:
\begin{prop}\label{prop1}
	Let $\sigma_0(s)$ be a translation compact function in $\Xi_+$. Let the family of semi-processes $\{U(t,\tau)\colon t\geq \tau\geq 0\}$, $\sigma \in \mathcal{H}_+(\sigma_0)$ be uniformly asymptotically compact and $\big(E\times \mathcal{H}_+(\sigma_0),E\big)$-continuous. Let $\omega(\mathcal{H}_+(\sigma_0))$ be the attractor of the translation semi-group $\{T(t)\}$ acting on $\mathcal{H}_+(\sigma_0)$. Consider the corresponding family of semi-processes $\{U_\sigma(t,\tau)\colon t\geq\tau\geq 0\}$, $\sigma \in \mathcal{H}_+(\sigma_0)$. Then, the uniform attractor $\mathcal{A}_{\mathcal{H}_+(\sigma_0)}$ of the family of semi-processes $\{U_\sigma(t,\tau)\colon t\geq\tau\geq 0\}$, $\sigma \in \mathcal{H}_+(\sigma_0)$ exists and coincides with the uniform attractor $\mathcal{A}_{\omega(\mathcal{H}_+(\sigma_0))}$ of the family of semi-processes $\{U_\sigma(t,\tau)\colon t\geq\tau\geq 0\}$, $\sigma \in \omega(\mathcal{H}_+(\sigma_0))$, i.e.
	\begin{equation*}
		\mathcal{A}_{\mathcal{H}_+(\sigma_0)} = \mathcal{A}_{\omega(\mathcal{H}_+(\sigma_0))}.
	\end{equation*}
\end{prop} 

\begin{proof}
	For the proof we refer the reader to \cite[Ch. 7, \S 3]{che}.
\end{proof}

The compactness of $\Sigma$ will be established in the framework of almost periodic functions and Stepanow spaces. We need a few more definitions (see \cite{Guter:1966ys} and \cite[Ch. 7, \S 5]{che}):
\begin{defi}
	We call the expression
	\begin{equation*}
		\ud_{S^p_l}\big(f(s),g(s)\big) = \sup_{t \geq 0} \left(\frac{1}{l}\int_{t}^{t + l}\norm{f(s) - g(s)}_{\Psi}^p\, \ud s \right)^{\frac{1}{p}}
	\end{equation*}
	the $S^p_l$-distance of the order $p \geq 1$ corresponding to the length $l$. The space of all $p$-power locally integrable functions with values in $\Psi$ and equipped with the norm generated by $\ud_{S^p_l}$ we call the \emph{Stepanow space} and denote by $L_{p}^S(\mathbb{R}^+;\Psi)$.
\end{defi}

In view of the above definition we set
\begin{equation*}
	\Xi_+ = L^S_{2,loc}(\mathbb{R}^+,\Psi).
\end{equation*}
Also note, that in the above definition we can put $l = 1$. Thus, we will write $S^p$ instead of $S^p_l$ or $S^p_1$. In addition, all $S_l^p$-spaces are complete. 

\begin{defi}
	A function $f \in L_p^S(\mathbb{R}^+;\Psi)$ we call \emph{$S^p$-asymptotically almost periodic} if for any $\epsilon \geq 0$ there exists a number $l = l(\epsilon)$ such that for each interval $(\alpha,\alpha + l)$, $\alpha \geq 0$ there exists a point $\tau \in (\alpha,\alpha + l)$ such that
	\begin{equation*}
		\ud_{S^p}\big(f(s + \tau),f(s)\big) = \sup_{t\geq 0} \left(\int_{t}^{t + 1} \norm{f(s + \tau) - f(s)}^p_{\Psi}\, \ud t\right)^{\frac{1}{p}} < \epsilon.
	\end{equation*}
\end{defi}

\begin{defi}

	We say that a function $f\in L_p^S(\mathbb{R}^+;\Psi)$ is 
	\emph{$S^p$-normal} if the family of translations
	\begin{equation*}
		\left\{ f^h(s) = f(h + s)\colon h \in \mathbb{R}^+\right\}
	\end{equation*}
	is precompact in $L_p^S(\mathbb{R}^+;\Psi)$ with respect to the norm 
	\begin{equation*}
		\left(\sup_{t \geq 0} \int_t^{t + 1} \norm{f(s)}^p_{\Psi}\, \ud s\right)^{\frac{1}{p}}
	\end{equation*}
	induced by $\ud_{S^p}$.
\end{defi}

From our point of view the following Proposition will play a crucial role:
\begin{prop}\label{prop3}
	A function $f$ is $S^p$-normal if and only if $f$ is $S^p$-asymptotically almost periodic. 
\end{prop}

Now we define
\begin{equation*}
	\mathcal{H}_+(\sigma_0) = \overline{\left\{\sigma^h_0(s) = \sigma_0(s + h)\colon h \geq 0\right\}}^{\Xi_+}.
\end{equation*}
Note, that the set $\mathcal{H}_+(\sigma_0)$ does not contain any information on the smallness assumption on the external data. Therefore, for any given $\sigma_0$ we 
additionally introduce 
\begin{equation*}
	\sigma^{\epsilon}_{0,x_3} = \big(f_{0,x_3},g_{0,x_3}\big),
\end{equation*}
where 
\begin{equation*}
	\norm{f_{0,x_3}}_{L_2(\Omega^t)}^2 + \norm{g_{0,x_3}}_{L_2(\Omega^t)}^2 < \epsilon
\end{equation*}
and define
\begin{equation*}
	\omega(\sigma_{0,x_3}^{\epsilon}) = \bigcap_{\tau \geq 0} \overline{\left\{\sigma_{0,x_3}^{\epsilon,h}(s)\colon h \geq \tau\right\}}^{\Xi_+}.
\end{equation*}
Then we set
\begin{equation*}
	\Sigma^\epsilon_{\sigma_0} = \mathcal{H}_+(\sigma_0) \cap \omega(\sigma_{0,x_3}^{\epsilon}).
\end{equation*}

\begin{rem}\label{rem8}
	In view of the above Proposition, the hull $\mathcal{H}_+(\sigma_0)$ is compact if we take $\sigma_0$ $S^2$-asymptotically almost periodic. But recall that we obtained the global existence of regular solutions under the assumption on the external data that they decay exponentially on every time interval of the form $[kT,(k + 1)T]$ (see Theorem \ref{thm0}). Thus, if we assume that
	\begin{equation*}
		\norm{f\big((k + 1)T + t\big) - f(kT + t)}_{\Psi} < \epsilon
	\end{equation*}
	(and so for $g$ and their derivatives with respect to $x_3$) then the external data become $S^2$-asympto\-tically almost periodic. In other words, we take almost the same external data on every time interval of the form $[kT,(k + 1)T]$. On the contrary, in the proof of the global existence of regular solutions, the external data could differ from interval to interval without any restrictions. 
\end{rem}

We can finally introduce the following Definition:
\begin{defi}
	Let $E$ be a Banach space (or complete metric or a closed subset of $E$). Let a two parameter family of operators 
	\begin{equation*}
		\left\{U_{\sigma}(t,\tau)\colon t\geq\tau\geq 0\right\}, \qquad U_{\sigma}(t,\tau)\colon E\to E
	\end{equation*}
	be given. We call it \emph{semi-process} in $E$ with the time symbol $\sigma \in \Sigma$ if
	\begin{description}
		\item[$(\mathbf{U})_1$] $U_{\sigma}(t,\tau) = U_{\sigma}(t,s)U_{\sigma}(s,\tau)$ for all $t\geq s\geq \tau \geq 0$, 
		\item[$(\mathbf{U})_2$] $U_{\sigma}(\tau,\tau) = \Id$ for all $\tau \geq 0$.
	\end{description}
\end{defi}

Observe that if we had a global and unique solution to problem \eqref{p1}, we could associate it with certain semi-process $\left\{U_{\sigma}(t,\tau)\colon t\geq\tau\geq 0\right\}$ defined on $\widehat H$. In three dimensions, likewise for the standard Navier-Stokes equations, we only know that such a semi-process would exist on $[0,t_{\max}]$, where $t_{\max} = t_{\max}(v_0,\omega_0,f,g)$. 

Following the idea from \cite[\S 2.3]{chol1} we can define a semi-process on certain subspace of $\widehat H$. Let $B^{\epsilon} \in \widehat H$ be such that
\begin{equation*}
	B^{\epsilon} = \left\{\big(v(0),\omega(0)\big) \in \widehat H\cap \widehat V\colon \norm{v(0)_{,x_3}}_{L_2(\Omega)}^2 + \norm{\Rot v(0)_{,x_3}}_{L_2(\Omega)}^2 + \norm{\omega(0)_{,x_3}}_{L_2(\Omega)}^2 < \epsilon\right\}.
\end{equation*}
Define
\begin{equation*}
	\omega_{\tau,\Sigma^{\epsilon}_{\sigma_0}}(B^{\epsilon}) := \bigcap_{t\geq\tau}\overline{\bigcup_{\sigma \in \Sigma^{\epsilon}_{\sigma_0}}\bigcup_{s\geq t}U_{\sigma_0}(s,\tau)B^{\epsilon}}^{\widehat H}
\end{equation*}
and introduce the sets
\begin{equation*}
	\begin{aligned}
		&\widehat H^{\epsilon} = \widehat H\cap \omega_{\tau,\Sigma^{\epsilon}_{\sigma_0}}(B^{\epsilon}), \\
		&\widehat V^{\epsilon} = \widehat V\cap \omega_{\tau,\Sigma^{\epsilon}_{\sigma_0}}(B^{\epsilon}).
	\end{aligned}
\end{equation*}

\begin{defi}
	The family of semiprocesses $\left\{U_{\sigma}(t,\tau)\right\}_{t\geq \tau\geq 0}$, $\sigma \in \Sigma^{\epsilon}_{\sigma_0}$ is said to be $(\widehat H^{\epsilon}\times \Sigma^{\epsilon}_{\sigma_0},\widehat H^{\epsilon})$-continuous if for any fixed $t,\tau \in \mathbb{R}_+$, $t \geq \tau$ the mapping $\big((v,\omega),\sigma\big)\mapsto U_{\sigma}(t,\tau)(v,\omega)$ is continuous from $\widehat H^{\epsilon}\times \Sigma^{\epsilon}_{\sigma_0}$ to $\widehat H^{\epsilon}$.
\end{defi}

Let us justify the $(\widehat H^{\epsilon}\times \Sigma^{\epsilon}_{\sigma_0},\widehat H^{\epsilon})$-continuity of $U_{\sigma}(v,\omega)_{t\geq\tau\geq 0}$. We have
\begin{lem}\label{lem35}
	The family $\left\{U_{\sigma}(t,\tau)\right\}_{t\geq\tau\geq 0}$, $\sigma \in \Sigma^{\epsilon}_{\sigma_0}$ of semiprocesses is $(\widehat H^{\epsilon}\times \Sigma^{\epsilon}_{\sigma_0},\widehat H^{\epsilon})$-continuous.
\end{lem}

\begin{proof}
	Let $\sigma_1 = (f^1,g^1)$ and $\sigma_2 = (f^2,g^2)$ be two different symbols. Consider two corresponding solutions $(v^1,\omega^1)$ and $(v^2,\omega^2)$ to problem \eqref{p1} with the initial conditions $(v^1_\tau,\omega^1_\tau)$ and $(v^2_\tau,\omega^2_\tau)$, respectively. Denote $V = v^1 - v^2$, $\Theta = \omega^1 - \omega^2$, $P = p^1 - p^2$, $F = f^1 - f^2$ and $G = g^1 - g^2$. Then the pair $(V,\Theta)$,
	\begin{equation*}
		(V,\Theta) = U_{\sigma_1}(t,\tau)(v^1_\tau,\omega^1_\tau) - U_{\sigma_2}(t,\tau)(v^2_\tau,\sigma^2_\tau),
	\end{equation*}
	is a solution to a following problem
	\begin{equation*}
		\begin{aligned}
				&V_{,t}  - (\nu + \nu_r)\triangle V + \nabla P = - V\cdot \nabla v^1 - v^2\cdot \nabla V + 2\nu_r\Rot\Theta + F& &\text{in } \Omega^t, \\
			&\Div V = 0 & &\text{in } \Omega^t,\\
			&\begin{aligned}
				&\Theta_{,t} - \alpha \triangle \Theta - \beta\nabla\Div\Theta + 4\nu_r \Theta \\
				&\mspace{60mu} = -v^1\cdot \nabla \Theta - V\cdot \nabla \omega^2 + 2\nu_r\Rot V + G
			\end{aligned}& &\text{in } \Omega^t,\\
			&\Rot V \times n = 0, \qquad V\cdot n = 0 & &\text{on } S^t, \\
			&\Theta = 0 & &\text{on } S^t_1, \\
			&\Theta' = 0, \qquad \Theta_{3,x_3} = 0 & &\text{on } S^t_2, \\
			&V\vert_{t = \tau} = V_\tau, \qquad \Theta\vert_{t = \tau} = \Theta_\tau & &\text{on } \Omega\times\{t = \tau\}.
		\end{aligned}
	\end{equation*}
	With $F$, $G$ and the initial conditions set to zero, we studied the above problem in \cite[see Lemma 11.1]{2012arXiv1205.4046N}. Thus, by analogous calculations we obtain the inequality
	\begin{multline*}
		\frac{1}{2}\Dt \int_{\Omega} \abs{V}^2 + \abs{\Theta}^2\, \ud x + \frac{\nu}{2c_{\Omega}}\norm{\Rot V}^2_{L_2(\Omega)} \\
		\leq \frac{c_{I,\Omega}}{\nu} \left(\norm{\nabla v^1}^2_{L_3(\Omega)} \norm{V}^2_{L_2(\Omega)} + \norm{\nabla \omega^2}^2_{L_3(\Omega)} \norm{\Theta}^2_{L_2(\Omega)} + \norm{F}^2_{L_{\frac{6}{5}}(\Omega)} + \norm{G}^2_{L_{\frac{6}{5}}(\Omega)}\right)
	\end{multline*}
	Application of the Gronwall inequality shows
	\begin{multline*}
		\norm{V(t)}^2_{L_2(\Omega)} + \norm{\Theta(t)}^2_{L_2(\Omega)} \leq c_{\nu,I,\Omega}\exp\left(\norm{\nabla v^1}^2_{L_2(\tau,t;L_3(\Omega))} + \norm{\nabla \omega^2}^2_{L_2(\tau,t;L_3(\Omega))}\right) \\
		\cdot \left(\norm{V(\tau)}^2_{L_2(\Omega)} + \norm{\Theta(\tau)}^2_{L_2(\Omega)} + \norm{F}^2_{L_2(\tau,t;L_{\frac{6}{5}}(\Omega))} + \norm{G}^2_{L_2(\tau,t;L_{\frac{6}{5}}(\Omega))}\right).
	\end{multline*}
	Since $H^2(\Omega) \hookrightarrow W^1_6(\Omega) \hookrightarrow W^1_3(\Omega)$ we see that
	\begin{align*}
		\norm{\nabla v^1}^2_{L_2(\tau,t;L_3(\Omega))} &\leq c_{\Omega} \norm{v^1}^2_{W^{2,1}_2(\Omega\times(\tau,t))}, \\
		\norm{\nabla \omega^2}^2_{L_2(\tau,t;L_3(\Omega))} &\leq c_{\Omega}\norm{\omega^2}^2_{W^{2,1}_2(\Omega\times(\tau,t))},
	\end{align*}

	By Theorem \ref{thm0} we get
	\begin{multline*}
		\norm{V(t)}^2_{L_2(\Omega)} + \norm{\Theta(t)}^2_{L_2(\Omega)} \\
		\leq c_{data}\left(\norm{V(\tau)}^2_{L_2(\Omega)} + \norm{\Theta(\tau)}^2_{L_2(\Omega)} + \norm{F}^2_{L_2(\tau,t;L_{\frac{6}{5}}(\Omega))} + \norm{G}^2_{L_2(\tau,t;L_{\frac{6}{5}}(\Omega))}\right),
	\end{multline*}
	which ends the proof.
\end{proof}

\section{Existence of the uniform attractor}\label{sec6}
In this section we prove the existence of the uniform attractor to problem \eqref{p1} restricted to $\widehat H^{\epsilon}$. By $\mathcal{B}(\widehat H)$ we denote the family of all bounded sets of $\widehat H$. We begin by introducing some fundamental definitions: 
\begin{defi}
	A family of processes $\{U_\sigma(t,\tau)\}_{t\geq \tau \geq 0}$, $\sigma \in \Sigma$ is said to be \emph{uniformly bounded} if for any $B \in \mathcal{B}(\widehat H)$ the set
	\begin{equation*}
		\bigcup_{\sigma\in\Sigma}\bigcup_{\tau\in\mathbb{R}^+}\bigcup_{t\geq\tau}U_\sigma(t,\tau)B\in\mathcal{B}(\widehat H).
	\end{equation*}
\end{defi}

\begin{defi}
	A set $B_0\in \widehat H^{\epsilon}$ is said to be \emph{uniformly absorbing} for the family of processes $\{U_\sigma(t,\tau)\}_{t\geq \tau \geq 0}$, $\sigma \in \Sigma$, if for any $\tau\in\mathbb{R}^+$ and for every $B\in\mathcal{B}(\widehat H)$ there exists $t_0 = t_0(\tau,B)$ such that $\bigcup_{\sigma\in\Sigma}U_\sigma(t,\tau)B\subseteq B_0$ for all $t\geq t_0$. If the set $B_0$ is compact, we call the family of processes \emph{uniformly compact}. 
\end{defi}

\begin{defi}\label{def6.3}
    A set $P$ belonging to $\widehat H$ is said to be \emph{uniformly attracting} for the family of processes $\{U_\sigma(t,\tau)\}_{t\geq \tau \geq 0}$, $\sigma \in \Sigma$, if for an arbitrary fixed $\tau \in \mathbb{R}^+$
	\begin{equation*}
		\lim_{t\to\infty} \left(\sup_{\sigma\in\Sigma}\dist_E \big(U_\sigma(t,\tau)B,P\big)\right) = 0.
	\end{equation*}
	If the set $P$ is compact, we call the family of processes \emph{uniformly asymptotically compact}. 
\end{defi}

\begin{defi}
	A closed set $\mathcal{A}_{\Sigma} \subset \widehat H$ is said to be the \emph{uniform attractor} of the family of processes $\{U_\sigma(t,\tau)\}_{t\geq \tau \geq 0}$, $\sigma \in \Sigma$, if it is uniformly attracting set and it is contained in any closed uniformly attracting set $\mathcal{A}'$ of the family processes $\{U_\sigma(t,\tau)\}_{t\geq \tau \geq 0}$, $\sigma \in \Sigma$, $\mathcal{A}_{\Sigma} \subseteq \mathcal{A}'$.
\end{defi}

To prove the existence of the uniform attractor we need two technical lemmas:

\begin{lem}\label{lem33}
	Let the assumptions of Theorem \ref{thm0} hold. Then, there exists a bounded and absorbing set in $\widehat H^{\epsilon}$ for the family of semiprocesses $\{U_\sigma(t,\tau)\}_{t\geq \tau \geq 0}$, $\sigma \in \Sigma^{\epsilon}_{\sigma_0}$.
\end{lem}
\begin{proof}
	Under the assumption of Theorem \ref{thm0} and in view assertion $(\mathbf{B})$ of \cite[Lemma 5.1]{2012arXiv1205.4507N} we see that for $t = (k + 1)T$, $t_0 = kT$ and sufficiently large $T > 0$
	\begin{equation*}
		\norm{v\big((k + 1)T\big)}_{L_2(\Omega)}^2 + \norm{\omega\big((k + 1)T\big)}_{L_2(\Omega)}^2 \leq \norm{v(kT)}_{L_2(\Omega)}^2 + \norm{\omega(kT)}_{L_2(\Omega)}^2
	\end{equation*}
	holds. By the same Lemma
	\begin{multline}\label{eq55}
		\limsup_{t\to\infty}\left(\norm{v(t)}_{L_2(\Omega)}^2 + \norm{\omega(t)}^2_{L_2(\Omega)}\right) \\
		\leq c_{\nu,\alpha,\Omega}\left(\sup_{k\in\mathbb{N}}\left(\norm{f(kT)}^2_{L_2(\Omega)} + \norm{g(kT)}^2_{L_2(\Omega)}\right) + \norm{v(0)}_{L_2(\Omega)}^2 + \norm{\omega(0)}_{L_2(\Omega)}^2\right) =: R_1.
	\end{multline}
	Thus, for every $(v_0,\omega_0)\in \widehat H^{\epsilon}$ there exists $t_0 > 0$ such that
	\begin{equation*}
		(v(t),\omega(t)) \in B(0,\rho_1) \qquad \forall_{t\geq t_0},
	\end{equation*}
	where $B(0,\rho_1)$ is the ball in $\widehat H^{\epsilon}$ centered at $0$ with radius $\rho_1 > R_1$. If $B(0,r)\subset \widehat H^{\epsilon}$ is any ball such that $(v_0,\omega_0) \in B(0,r)$ then there exists $t_0  =t_0(r)$ such that \eqref{eq55} holds. This proves the existence of bounded absorbing sets in $\widehat H^{\epsilon}$ for the semiprocess $\{U_\sigma(t,\tau)\}_{t\geq \tau \geq 0}$, $\sigma \in \Sigma^{\epsilon}_{\sigma_0}$.
\end{proof}

\begin{lem}\label{lem34}
	Let the assumptions of Theorem \ref{thm0} hold. Then, there exists a bounded and absorbing set in $\widehat V^{\epsilon}$ for the family of semiprocesses $\{U_\sigma(t,\tau)\}_{t\geq \tau \geq 0}$, $\sigma \in \Sigma^{\epsilon}_{\sigma_0}$.
\end{lem}

\begin{proof}
	From the assumptions of Theorem \ref{thm0} and from \cite[Lemma 5.2]{2012arXiv1205.4507N}, where we set $t_1 = (k + 1)T$, $t_0 = kT$, we infer that
	\begin{multline}\label{eq12}
		\limsup_{t \to \infty} \left(\norm{v(t)}^2_{H^1(\Omega)} + \norm{v(t)}^2_{H^1(\Omega)}\right) \\
		\leq c_{\nu,\alpha,\beta,I,P,\Omega}\left(\sup_{k\in\mathbb{N}}\left(\norm{f(kT)}^2_{L_2(\Omega)} + \norm{g(kT)}^2_{L_2(\Omega)}\right) + \norm{v(0)}_{H^1(\Omega)}^2 + \norm{\omega(0)}_{H^1(\Omega)}^2\right) =: R_2.
	\end{multline}
	Likewise in previous Lemma, for every $(v_0,\omega_0)\in \widehat V^{\epsilon}$ there exists $t_0 > 0$ such that
	\begin{equation*}
		(v(t),\omega(t)) \in B(0,\rho_2) \qquad \forall_{t\geq t_0},
	\end{equation*}
	where $B(0,\rho_2)$ is the ball in $\widehat V^{\epsilon}$ centered at $0$ with radius $\rho_2 > R_2$. For any $(v_0,\omega_0) \in B(0,r) \subset \widehat V^{\epsilon}$ there exists $t_0  =t_0(r)$ such that \eqref{eq12} holds. Therefore there exist bounded absorbing sets in $\widehat V^{\epsilon}$ for the semiprocess $\{U_\sigma(t,\tau)\}_{t\geq \tau \geq 0}$, $\sigma \in \Sigma^{\epsilon}_{\sigma_0}$.
\end{proof}

\begin{proof}[Proof of Theorem \ref{thm1}]
	To prove the existence of the uniform attractor we make use of Proposition \ref{prop1}. We need to check the assumptions:
	\begin{itemize}
		\item	From Remark \ref{rem8} and by assumption on the data it follows that $\sigma_0 = (f_0,g_0) \in \Xi_+$ is translation compact.
		\item From Definition \ref{def6.3} it follows that the family of semiprocesses $\{U_\sigma(t,\tau)\}_{t\geq \tau \geq 0}$, $\sigma \in \Sigma^{\epsilon}_{\sigma_0}$ is uniformly asymptotically compact, if there exists a set $P \in \widehat H^{\epsilon}$, which is uniformly attracting and compact. Lemmas \ref{lem33} and \ref{lem34} and the Sobolev embedding theorem ensure the existence of such set $P$.
		\item The $(\widehat H^{\epsilon}\times \Sigma^{\epsilon}_{\sigma_0},\widehat H^{\epsilon})$-continuity of the family of semiprocesses $\{U_\sigma(t,\tau)\}_{t\geq \tau \geq 0}$, $\sigma \in \Sigma^{\epsilon}_{\sigma_0}$ was proved in Lemma \ref{lem35}.
		\item To check that $\omega(\Sigma_{\sigma_0}^{\epsilon})$ is the global attractor for the translation semigroup $\{T(t)\}_{t\geq 0}$ acting on $\Sigma_{\sigma_0}^{\epsilon}$ we recall that $\Sigma_{\sigma_0}^{\epsilon} \Subset \Xi_+$ is compact metric space. We can therefore apply Proposition \ref{prop2}. Thus,
			\begin{align*}
				\dist_{\Xi_+} \big(T(t)\Sigma_{\sigma_0}^{\epsilon},\omega(\Sigma_{\sigma_0}^{\epsilon})\big) \xrightarrow[t\to 0]{}0
				\intertext{and}
				T(t)\omega(\Sigma_{\sigma_0}^{\epsilon}) = \omega(\Sigma_{\sigma_0}^{\epsilon}) \qquad \forall_{t\geq 0}.
		\end{align*}
	\end{itemize}
	As we see all assumptions are satisfied. The application of Proposition \ref{prop1} ends the proof.
\end{proof}

\section{Convergence to stationary solutions for large $\nu$}\label{sec7}

In this section we prove Theorem \ref{thm3}. 

\begin{proof}
    The existence of solutions for large $\nu$ and theirs estimate follows, after slight modifications, from \cite{luk0} and \cite[Ch. 2, \S 1, Theorems 1.1.1, 1.1.2 and 1.1.3]{luk1}.
	
	Let $V = v(t)-v_\infty$, $\Theta = \omega(t) - \omega_{\infty}$ and $P = p(t) - p_{\infty}$. Then $(V,\Theta)$ satisfies
	\begin{equation}\label{eq38}
		\begin{aligned}
			&V_{,t} - (\nu + \nu_r)\triangle V + \nabla P = -v\cdot\nabla V - V\cdot\nabla v_\infty + 2\nu_r \Rot \Theta  & &\text{in $\Omega^t$}, \\
			&\Div V = 0 & &\text{in $\Omega^t$}, \\
			&\begin{aligned}
				&\Theta_{,t} - \alpha \triangle \Theta - \beta\nabla\Div\Theta + 4\nu_r \Theta \\
				&\mspace{60mu} = - v\cdot\nabla \Theta - V\cdot \nabla \omega_{\infty} + 2\nu_r \Rot V 
			\end{aligned} & &\text{in $\Omega^t$}, \\
			&\Rot V\times n = 0 & &\text{on $S^t$}, \\
			&V\cdot n  = 0 & &\text{on $S^t$}, \\
			&\Theta = 0 & &\text{on $S_1^t$}, \\
			&\Theta' = 0, \qquad \Theta_{3,x_3} = 0 & &\text{on $S_2^t$}, \\
			&V\vert_{t = 0} = v(0) - v_\infty, \qquad \Theta\vert_{t = 0} = \omega(0) - \omega_{\infty} & &\text{in $\Omega\times\{t = 0\}$}.
		\end{aligned}
	\end{equation}
	Multiplying the first and the third equation by $V$ and $\Theta$ respectively and integrating over $\Omega$ yields
	\begin{multline*}
		\frac{1}{2}\Dt \norm{V}^2_{L_2(\Omega)} + \left(\nu + \nu_r\right)\norm{\Rot V}_{L_2(\Omega)}^2 + \int_S \Rot V\times n \cdot V\, \ud S \\
		= -\int_{\Omega} V\cdot \nabla v_{\infty} \cdot V\, \ud x + 2\nu_r\int_{\Omega} \Rot V \cdot \Theta\, \ud x + \int_S \Theta \times n \cdot V\, \ud S
	\end{multline*}
	and
	\begin{multline*}
		\frac{1}{2}\Dt \norm{\Theta}^2_{L_2(\Omega)} + \alpha\norm{\Rot\Theta}^2_{L_2(\Omega)} + (\alpha + \beta)\norm{\Div \Theta}^2_{L_2(\Omega)} + 4\nu_r\norm{\Theta}^2_{L_2(\Omega)} \\
		= -\int_{\Omega} V\cdot \nabla \omega_{\infty}\cdot \Theta\, \ud x + 2\nu_r\int_{\Omega} \Rot V \cdot \Theta\, \ud x, 
	\end{multline*}
	because in view of \eqref{eq38}$_{2,5}$ and \eqref{p1}$_{2,4}$ we see that
	\begin{equation*}
		\begin{aligned}
			&\int_{\Omega}\nabla P\cdot V\, \ud x = \int_S P (V\cdot n)\, \ud S = 0, \\
			&-\int_{\Omega}v\cdot \nabla V\cdot V\, \ud x = -\frac{1}{2}\int_{\Omega}v\cdot \nabla\abs{V}^2\, \ud x = -\int_S \abs{V}^2 (v\cdot n)\, \ud S = 0, \\
			&-\int_{\Omega}v\cdot \nabla \Theta\cdot \Theta\, \ud x = -\frac{1}{2}\int_{\Omega}v\cdot \nabla\abs{\Theta}^2\, \ud x = -\int_S \abs{\Theta}^2 (v\cdot n)\, \ud S = 0.
		\end{aligned}
	\end{equation*}
	Adding both equalities gives
	\begin{multline}\label{eq39}
		\frac{1}{2}\Dt \left(\norm{V}^2_{L_2(\Omega)} + \norm{\Theta}^2_{L_2(\Omega)}\right) + \left(\nu + \nu_r\right)\norm{\Rot V}_{L_2(\Omega)}^2 + \alpha\norm{\Rot\Theta}^2_{L_2(\Omega)} \\
		+ (\alpha + \beta)\norm{\Div \Theta}^2_{L_2(\Omega)} + 4\nu_r\norm{\Theta}^2_{L_2(\Omega)} \\
		= 4\nu_r\int_{\Omega} \Rot V \cdot \Theta\, \ud x -\int_{\Omega} V\cdot \nabla v_{\infty} \cdot V\, \ud x -\int_{\Omega} V\cdot \nabla \omega_{\infty}\cdot \Theta\, \ud x =: I_1 + I_2 + I_3
	\end{multline}
	Utilizing the H\"older and the Young inequalities on the right-hand side we have for $I_1$
	\begin{equation*}
			I_1 \leq 4\nu_r \norm{\Rot V}_{L_2(\Omega)}\norm{\Theta}_{L_2(\Omega)} \leq 4\nu_r\epsilon_1\norm{\Rot V}_{L_2(\Omega)} + \frac{\nu_r}{\epsilon_1}\norm{\Theta}_{L_2(\Omega)}.
	\end{equation*}
	To estimate $I_2$ and $I_3$ we also make use of the interpolation inequality for $L_p$-spaces and the embedding theorem:
	\begin{equation*}
		\begin{aligned}
			\norm{u}_{L_4(\Omega)} \leq \norm{u}_{L_2(\Omega)}^{\frac{1}{4}}\norm{u}_{L_6(\Omega)}^{\frac{3}{4}} \leq c_I  \norm{u}_{L_2(\Omega)}^{\frac{1}{4}}\norm{u}_{H^1(\Omega)}^{\frac{3}{4}}, \\
			\norm{u}_{L_3(\Omega)} \leq \norm{u}_{L_2(\Omega)}^{\frac{1}{2}}\norm{u}_{L_6(\Omega)}^{\frac{1}{2}} \leq c_I \norm{u}_{L_2(\Omega)}^{\frac{1}{2}}\norm{u}_{H^1(\Omega)}^{\frac{1}{2}}.
		\end{aligned}
	\end{equation*}
	Thus
	\begin{multline*}
		I_2 \leq \norm{V}_{L_4(\Omega)}\norm{\nabla v_{\infty}}_{L_2(\Omega)}\norm{V}_{L_4(\Omega)} \leq c_I\norm{V}_{L_2(\Omega)}^{\frac{1}{2}} \norm{V}_{H^1(\Omega)}^{\frac{3}{2}}\norm{v_{\infty}}_{H^1(\Omega)} \\
		\leq \epsilon_2c_I \norm{V}_{H^1(\Omega)}^2 + \frac{c_I}{4\epsilon_2} \norm{V}_{L_2(\Omega)}^2\norm{v_{\infty}}_{H^1(\Omega)}^4
	\end{multline*}
	and
	\begin{multline*}
		I_3 \leq \norm{V}_{L_6(\Omega)} \norm{\nabla \omega_{\infty}}_{L_2(\Omega)} \norm{\Theta}_{L_3(\Omega)} \leq c_I\norm{V}_{H^1(\Omega)}\norm{\omega_{\infty}}_{H^1(\Omega)} \norm{\Theta}_{L_2(\Omega)}^{\frac{1}{2}} \norm{\Theta}_{H^1(\Omega)}^{\frac{1}{2}} \\
		\leq c_I \epsilon_{31}\norm{V}_{H^1(\Omega)}^2 + \frac{c_I}{4\epsilon_{31}}\norm{\omega_{\infty}}_{H^1(\Omega)}^2\norm{\Theta}_{L_2(\Omega)}\norm{\Theta}_{H^1(\Omega)} \\
		\leq c_I \epsilon_{31}\norm{V}_{H^1(\Omega)}^2 + \frac{c_I\epsilon_{32}}{4\epsilon_{31}} \norm{\Theta}^2_{H^1(\Omega)} + \frac{c_I}{16\epsilon_{31}\epsilon_{32}}\norm{\omega_{\infty}}_{H^1(\Omega)}^4\norm{\Theta}_{L_2(\Omega)}^2.
	\end{multline*}
	Now we set $\epsilon_1 = \frac{1}{4}$. By \cite[Lemma 6.7]{2012arXiv1205.4046N} it follows that
	\begin{equation*}
		\nu\norm{\Rot V}_{L_2(\Omega)}^2 + \alpha\norm{\Rot\Theta}^2_{L_2(\Omega)} + (\alpha + \beta)\norm{\Div \Theta}^2_{L_2(\Omega)} \geq \frac{\nu}{c_{\Omega}} \norm{V}_{H^1(\Omega)}^2 + \frac{\alpha}{c_{\Omega}}\norm{\Theta}^2_{H^1(\Omega)}.
	\end{equation*}
	It enables us to put $\epsilon_2 c_I = \epsilon_{31} c_I = \frac{\nu}{4c_{\Omega}}$ and $\epsilon_{32} \frac{c_I}{4\epsilon_{31}} =  \frac{\alpha}{2c_{\Omega}}$. Hence
	\begin{align*}
		\epsilon_2 = \frac{\nu}{4c_{I,\Omega}}, & &\epsilon_{31} = \frac{\nu}{4c_{I,\Omega}}, & &\epsilon_{32} = \frac{\alpha\nu}{3c_{I,\Omega}}, & \\
		\frac{c_I}{4\epsilon_2} = \frac{c_{I,\Omega}}{\nu}, & & & &\frac{c_I}{16\epsilon_{31}\epsilon_{32}} = \frac{3c_{I,\Omega}}{4\alpha\nu^2}. &
	\end{align*}
	Therefore, from \eqref{eq39} we obtain
	\begin{multline*}
		\frac{1}{2}\Dt\left(\norm{V}^2_{L_2(\Omega)} + \norm{\Theta}_{L_2(\Omega)}^2\right) + \frac{\nu}{2c_{\Omega}}\norm{V}_{H^1(\Omega)}^2 + \frac{\alpha}{2c_{\Omega}}\norm{\Theta}^2_{H^1(\Omega)} \\
		\leq \frac{c_{I,\Omega}}{\nu}\norm{V}_{L_2(\Omega)}^2\norm{v_\infty}^4_{H^1(\Omega)} +  \frac{3c_{I,\Omega}}{4\alpha\nu^2}\norm{\Theta}_{L_2(\Omega)}^2\norm{\omega_\infty}^4_{H^1(\Omega)}.
	\end{multline*}
	Let
	\begin{equation*}
		\begin{aligned}
			c_1(\nu) &= \frac{\min\{\nu,\alpha\}}{c_{\Omega}}, \\
			c_2 &= \frac{c_{I,\Omega}}{\min\{1,\alpha\}}
		\end{aligned}
	\end{equation*}
	and define
	\begin{equation*}
		\Delta(\nu) = c_1(\nu) - \frac{3c_2}{\nu}F^4\big(\norm{f_\infty}_{L_2(\Omega)},\norm{g_\infty}_{L_2(\Omega)}\big).
	\end{equation*}
	Observe, that $\Delta(\nu) \to \frac{\alpha}{c_{\Omega}} > 0$ as $\nu \to \infty$. In particular there exists $\nu_* > 0$ such that $\Delta(\nu) > 0$ for any $\nu > \nu_*$ holds. The last inequality implies
	\begin{equation*}
		\Dt\left(\norm{V}^2_{L_2(\Omega)} + \norm{\Theta}_{L_2(\Omega)}^2\right) + \Delta(\nu)\left(\norm{V}_{L_2(\Omega)}^2 + \norm{\Theta}^2_{L_2(\Omega)}\right) \leq 0,
	\end{equation*}
	which is equivalent to
	\begin{equation*}
		\Dt\left(\left(\norm{V}^2_{L_2(\Omega)} + \norm{\Theta}_{L_2(\Omega)}^2\right)e^{\Delta(\nu)t}\right) \leq 0.
	\end{equation*}
	Integration with respect to $t$ yields
	\begin{equation*}
		\left(\norm{V(t)}^2_{L_2(\Omega)} + \norm{\Theta(t)}_{L_2(\Omega)}^2\right)e^{\Delta(\nu)t} \leq \norm{V(0)}^2_{L_2(\Omega)} + \norm{\Theta(0)}_{L_2(\Omega)}^2,
	\end{equation*}
	which implies
	\begin{equation*}
		\norm{V(t)}^2_{L_2(\Omega)} + \norm{\Theta(t)}_{L_2(\Omega)}^2 \leq \left(\norm{V(0)}^2_{L_2(\Omega)} + \norm{\Theta(0)}_{L_2(\Omega)}^2\right)e^{-\Delta(\nu)t}.
	\end{equation*}
	This concludes the proof.  
\end{proof}

\section{Continuous dependence on modeling}\label{sec8}

In this Section we examine the difference between $(v,\omega)$ and the solution $(u,\Theta)$ to problem \eqref{p11}.

Observe, that \eqref{p11} is the same as \eqref{p1} with $\nu_r = 0$. Therefore, for $\nu_r$ close to zero we may measure in some sense the deviation of the flows of micropolar fluids from that of modeled by the Navier-Stokes equations. This problem was considered locally in \cite{paye} and globally in \cite[\S 5]{luk2}. 

\begin{proof}[Proof of Theorem \ref{thm2}]
	Let us denote $V(t) = v(t) - u(t)$, $Z(t) = \omega(t) - \Theta(t)$. Then the pair $(V(t),Z(t))$ is a solution to the problem
	\begin{equation*}
		\begin{aligned}
			&V_{,t} - (\nu + \nu_r)\triangle V - \nu_r \triangle u + \nabla (p - q) = -v\cdot\nabla V - V\cdot\nabla u + 2\nu_r \Rot \omega  & &\text{in $\Omega^t$}, \\
			&\Div V = 0 & &\text{in $\Omega^t$}, \\
			&\begin{aligned}
				&Z_{,t} - \alpha \triangle Z - \beta\nabla\Div Z + 4\nu_r \omega \\
				&\mspace{60mu} = - v\cdot\nabla Z - V\cdot \nabla \Theta + 2\nu_r \Rot v 
			\end{aligned} & &\text{in $\Omega^t$}, \\
			&\Rot V\times n = 0 & &\text{on $S^t$}, \\
			&V\cdot n  = 0 & &\text{on $S^t$}, \\
			&Z = 0 & &\text{on $S_1^t$}, \\
			&Z' = 0, \qquad Z_{3,x_3} = 0 & &\text{on $S_2^t$}, \\
			&V\vert_{t = t_0} = v(t_0) - u(t_0), \qquad Z\vert_{t = t_0} = \omega(t_0) - \Theta(t_0) & &\text{in $\Omega\times\{t = t_0\}$}.
		\end{aligned}
	\end{equation*}
	Multiplying the first equation and the third by $V$ and $Z$, respectively and integrating over $\Omega$ yields
	\begin{multline*}
		\frac{1}{2}\Dt \int_{\Omega} \abs{V}^2\, \ud x + (\nu + \nu_r)\int_{\Omega}\abs{\Rot V}^2\, \ud x + (\nu + \nu_r)\int_S \Rot V \times n \cdot V\, \ud S + \nu_r\int_{\Omega} \Rot u \cdot \Rot V\, \ud x \\
		+ \nu_r \int_S \Rot u\times n \cdot V\, \ud S + \int_S (p - q) V \cdot n\, \ud S \\
		= -\int_{\Omega} v\cdot \nabla V\cdot V\, \ud x - \int_{\Omega} V\cdot \nabla u \cdot V\, \ud x + 2\nu_r \int_{\Omega} \Rot \omega\cdot V\, \ud x
	\end{multline*}
	and
	\begin{multline*}
		\frac{1}{2}\Dt \int_{\Omega} \abs{Z}^2\, \ud x + \alpha\int_{\Omega} \abs{\Rot Z}^2\, \ud x + (\alpha + \beta)\int_{\Omega} \abs{\Div Z}^2\, \ud x - \alpha\int_S Z\times n \cdot \Rot Z\, \ud S + 4\nu_r \int_{\Omega}\omega\cdot Z\, \ud x \\
		= -\int_{\Omega} v\cdot \nabla Z\cdot Z\, \ud x - \int_{\Omega} V \cdot \nabla \Theta \cdot Z\, \ud x + 2\nu_r\int_{\Omega} \Rot v\cdot Z\, \ud x.
	\end{multline*}
	All boundary integrals vanish due to the boundary conditions. In addition
	\begin{equation*}
		\begin{aligned}
			-\int_{\Omega} v\cdot \nabla V\cdot V\, \ud x &= -\frac{1}{2}\int_{\Omega} v\cdot \nabla \abs{V}^2\, \ud x = -\frac{1}{2}\int_S \abs{V}^2 v \cdot n \, \ud S = 0, \\
			-\int_{\Omega} v\cdot \nabla Z\cdot Z\, \ud x &= -\frac{1}{2}\int_{\Omega} v\cdot \nabla \abs{Z}^2\, \ud x = -\frac{1}{2}\int_S \abs{Z}^2 v \cdot n \, \ud S = 0.
		\end{aligned}
	\end{equation*}
	We also have
	\begin{equation*}
		2\nu_r\int_{\Omega} \Rot \omega\cdot V\, \ud x = 2\nu_r\int_{\Omega} \omega\cdot \Rot V\, \ud x + 2\nu_r\int_S \omega\times n \cdot V\, \ud S = 2\nu_r\int_{\Omega} \omega\cdot \Rot V\, \ud x,
	\end{equation*}
	which follows from integration by parts and the boundary conditions. Thus, we get
	\begin{multline*}
		\frac{1}{2}\Dt\left(\norm{V}^2_{L_2(\Omega)} + \norm{Z}^2_{L_2(\Omega)}\right) + (\nu + \nu_r)\norm{\Rot V}^2_{L_2(\Omega)} + \alpha\norm{\Rot Z}^2_{L_2(\Omega)} + (\alpha + \beta)\norm{\Div Z}^2_{L_2(\Omega)} \\
		= -\nu_r\int_{\Omega} \Rot u \cdot \Rot V\, \ud x- \int_{\Omega} V\cdot \nabla u \cdot V\, \ud x + 2\nu_r \int_{\Omega} \omega\cdot \Rot V\, \ud x \\
		- 4\nu_r \int_{\Omega}\omega\cdot Z\, \ud x - \int_{\Omega} V \cdot \nabla \Theta \cdot Z\, \ud x + 2\nu_r\int_{\Omega} \Rot v\cdot Z\, \ud x.
	\end{multline*}
	Next we estimate the first and the third term on the right-hand side by the means of the H\"older and the Young inequalities. We have
	\begin{equation*}
		\begin{aligned}
			-\nu_r\int_{\Omega} \Rot u \cdot \Rot V\, \ud x &\leq \nu_r\norm{\Rot u}_{L_2(\Omega)}\norm{\Rot V}_{L_2(\Omega)} \leq \epsilon_1\nu_r\norm{\Rot V}_{L_2(\Omega)}^2 + \frac{\nu_r}{4\epsilon_1}\norm{\Rot u}_{L_2(\Omega)}, \\
			2\nu_r \int_{\Omega} \omega\cdot \Rot V\, \ud x &\leq 2\nu_r\norm{\omega}_{L_2(\Omega)}\norm{\Rot V}_{L_2(\Omega)} \leq 2\epsilon_2\nu_r\norm{\Rot V}_{L_2(\Omega)}^2 + \frac{\nu_r}{2\epsilon_2}\norm{\omega}_{L_2(\Omega)}.
		\end{aligned}
	\end{equation*}
	We set $\epsilon_1 \nu_r = 2\epsilon_2 \nu_r = \frac{\nu_r}{2}$. Thus, $\epsilon_1 = \frac{1}{2}$ and $\epsilon_2 = \frac{1}{4}$. We get
	\begin{multline*}
		\frac{1}{2}\Dt\left(\norm{V}^2_{L_2(\Omega)} + \norm{Z}^2_{L_2(\Omega)}\right) + \nu\norm{\Rot V}^2_{L_2(\Omega)} + \alpha\norm{\Rot Z}^2_{L_2(\Omega)} + (\alpha + \beta)\norm{\Div Z}^2_{L_2(\Omega)} \\
		\leq - \int_{\Omega} V\cdot \nabla u \cdot V\, \ud x  - 4\nu_r \int_{\Omega}\omega\cdot Z\, \ud x - \int_{\Omega} V \cdot \nabla \Theta \cdot Z\, \ud x + 2\nu_r\int_{\Omega} \Rot v\cdot Z\, \ud x.
	\end{multline*}
	Utilizing \cite[Lemma 6.7]{2012arXiv1205.4046N} on the left-hand side yields
	\begin{multline*}
		\frac{1}{2}\Dt\left(\norm{V}^2_{L_2(\Omega)} + \norm{Z}^2_{L_2(\Omega)}\right) + \frac{\nu}{c_{\Omega}}\norm{V}^2_{H^1(\Omega)} + \frac{\alpha}{c_{\Omega}}\norm{Z}^2_{H^1(\Omega)} \\
		\leq - \int_{\Omega} V\cdot \nabla u \cdot V\, \ud x  - 4\nu_r \int_{\Omega}\omega\cdot Z\, \ud x - \int_{\Omega} V \cdot \nabla \Theta \cdot Z\, \ud x + 2\nu_r\int_{\Omega} \Rot v\cdot Z\, \ud x.
	\end{multline*}
	For the first term we have
	\begin{multline*}
		- \int_{\Omega} V\cdot \nabla u \cdot V\, \ud x \leq \norm{V}_{L_4(\Omega)}^2\norm{\nabla u}_{L_2(\Omega)} \leq c_I\norm{V}_{L_6(\Omega)}^{\frac{3}{2}}\norm{V}_{L_2(\Omega)}^{\frac{1}{2}}\norm{u}_{H^1(\Omega)} \\
		\leq \epsilon_1 c_I\norm{V}_{H^1(\Omega)}^2 + \frac{c_I}{4\epsilon_1} \norm{V}_{L_2(\Omega)}^2\norm{u}_{H^1(\Omega)}^4,
	\end{multline*}
	where we used the H\"older inequality, the interpolation inequality between $L_2(\Omega)$ and $L_6(\Omega)$ and the Young inequality.
	
	For the second term we get
	\begin{equation*}
		- 4\nu_r \int_{\Omega}\omega\cdot Z\, \ud x \leq 4\nu_r\norm{\omega}_{L_2(\Omega)}	\norm{Z}_{L_2(\Omega)} \leq 4\nu_r\epsilon_2\norm{\omega}_{L_2(\Omega)}^2 + \frac{\nu_rc_I}{\epsilon_2}\norm{Z}_{H^1(\Omega)}^2.
	\end{equation*}
	
	The third term we estimate in the following way
	\begin{equation*}
		- \int_{\Omega} V \cdot \nabla \Theta \cdot Z\, \ud x \leq \norm{V}_{L_6(\Omega)}\norm{\nabla \Theta}_{L_3(\Omega)}\norm{Z}_{L_2(\Omega)} \leq c_I\epsilon_3\norm{V}_{H^1(\Omega)}^2 + \frac{1}{4\epsilon_3}\norm{\nabla \Theta}_{L_3(\Omega)}^2\norm{Z}_{L_2(\Omega)}^2.
	\end{equation*}
	
	Finally, for the fourth term we have
	\begin{equation*}
		2\nu_r \int_{\Omega}\Rot v\cdot Z\, \ud x \leq 2\nu_r\norm{\Rot v}_{L_2(\Omega)} \norm{Z}_{L_2(\Omega)} \leq 2\nu_r\epsilon_4\norm{\Rot v}_{L_2(\Omega)}^2 + \frac{\nu_rc_I}{2\epsilon_4}\norm{Z}_{H^1(\Omega)}^2.
	\end{equation*}
	
	We set $\epsilon_1 c_I = \epsilon_3 c_I = \frac{\nu}{4c_{\Omega}}$ and $\frac{\nu_rc_I}{\epsilon_2} = \frac{\nu_rc_I}{2\epsilon_4} = \frac{\alpha}{4c_{\Omega}}$. Hence
	\begin{align*}
		\frac{c_I}{4\epsilon_1} &= \frac{c_{I,\Omega}}{\nu}, & & & 4\nu_r\epsilon_2 &= \frac{16\nu_r^2c_{I,\Omega}}{\alpha}, & & & \frac{1}{4\epsilon_3} &= \frac{c_{I,\Omega}}{\nu}, & & & 4\nu_r\epsilon_4 &= \frac{8\nu_r^2c_{I,\Omega}}{\alpha}
	\end{align*}
	and we obtain
	\begin{multline*}
		\frac{1}{2}\Dt\left(\norm{V}^2_{L_2(\Omega)} + \norm{Z}^2_{L_2(\Omega)}\right) + \frac{\nu}{2c_{\Omega}}\norm{V}^2_{H^1(\Omega)} + \frac{\alpha}{2c_{\Omega}}\norm{Z}^2_{H^1(\Omega)} \\
		\leq \frac{c_{I,\Omega}}{\nu}\norm{V}_{L_2(\Omega)}^2\norm{u}_{H^1(\Omega)}^4 + \frac{16\nu_r^2c_{I,\Omega}}{\alpha}\norm{\omega}_{L_2(\Omega)}^2 \\
		+ \frac{c_{I,\Omega}}{\nu}\norm{\nabla \Theta}_{L_3(\Omega)}^2\norm{Z}_{L_2(\Omega)}^2 + \frac{8\nu_r^2c_{I,\Omega}}{\alpha}\norm{\Rot v}_{L_2(\Omega)}^2.
	\end{multline*}
	We introduce
	\begin{equation}\label{eq44}
		\begin{aligned}
			\Delta_1(\nu) &= \frac{\nu}{c_{\Omega}} - \frac{2c_{I,\Omega}}{\nu}\norm{u}_{H^1(\Omega)}^4, \\
			\Delta_2(\nu) &= \frac{\alpha}{c_{\Omega}} - \frac{2c_{I,\Omega}}{\nu}\norm{\nabla \Theta}_{L_3(\Omega)}^2, \\
			\Delta(\nu) &= \min\left\{\Delta_1(\nu),\Delta_2(\nu)\right\}.
		\end{aligned}
	\end{equation}
	Thus, the last inequality implies
	\begin{multline*}
		\Dt\left(\norm{V}^2_{L_2(\Omega)} + \norm{Z}^2_{L_2(\Omega)}\right) + \Delta(\nu)\left(\norm{V}^2_{L_2(\Omega)} + \norm{Z}^2_{L_2(\Omega)}\right) \\
		\leq \frac{32\nu_r^2c_{I,\Omega}}{\alpha}\norm{\omega}_{L_2(\Omega)}^2 + \frac{16\nu_r^2c_{I,\Omega}}{\alpha}\norm{\Rot v}_{L_2(\Omega)}^2,
	\end{multline*}
	which is equivalent to
	\begin{equation*}
		\Dt\left(\left(\norm{V}^2_{L_2(\Omega)} + \norm{Z}^2_{L_2(\Omega)}\right)e^{\Delta(\nu)t}\right) \leq \nu_r^2c_{\alpha,I,\Omega}\left(\norm{\omega}_{L_2(\Omega)}^2 + \norm{\Rot v}_{L_2(\Omega)}^2\right)e^{\Delta(\nu)t}.
	\end{equation*}
	Integrating with respect to $t \in (t_0,t_1)$ yields
	\begin{multline*}
		\norm{V(t)}^2_{L_2(\Omega)} + \norm{Z(t)}^2_{L_2(\Omega)} \\
		\leq \nu_r^2c_{\alpha,I,\Omega}\left(\norm{\omega}_{L_2(\Omega^t)}^2 + \norm{\Rot v}_{L_2(\Omega^t)}^2\right) + \left(\norm{V(t_0)}^2_{L_2(\Omega)} + \norm{Z(t_0)}^2_{L_2(\Omega)}\right)e^{-\Delta(\nu)t}.
	\end{multline*}
	In view of the energy estimates (see e.g. \cite[Lemma 8.1]{2012arXiv1205.4046N}) we obtain
	\begin{equation*}
		\norm{\omega}_{L_2(\Omega^t)}^2 + \norm{\Rot v}_{L_2(\Omega^t)}^2 < E_{v,\omega}(t).
	\end{equation*}
	
	On the other hand, under slightly weaker assumption than in Theorem \ref{thm0} it follows from \cite[Theorem B]{wm6} that
	\begin{equation}\label{eq10}
		\norm{u}_{W^{2,1}_2(\Omega^t)} \leq \varphi(\nu,\text{data})
	\end{equation}
	for any $t \in [0,\infty]$, where $\varphi$ is a positive and non-decreasing function dependent on the viscosity coefficient and the initial and the external data. By the Embedding Theorem for anisotropic Sobolev spaces (see e.g. \cite[Ch. 2, \S 3, Lemma 3.3]{lad} it follows that $W^{2,1}_2(\Omega^t) \hookrightarrow L_{10}(\Omega^t)$ and
	\begin{equation*}
		\norm{u}_{L_{10}(\Omega^t)} \leq c_{\Omega} \norm{u}_{W^{2,1}_2(\Omega^t)} \leq \varphi(\text{data})
	\end{equation*}
	holds. In light of the maximal regularity for parabolic systems (see e.g.  \cite[Lemma 6.9]{2012arXiv1205.4046N}) applied to \eqref{p11}$_3$ we deduce that $\Theta \in W^{2,1}_{\frac{5}{3}}(\Omega^t)$ and
	\begin{equation*}
		\norm{\Theta}_{W^{2,1}_{\frac{5}{3}}(\Omega^t)} \leq c_{\text{data},\alpha,\beta,\Omega} \left(\norm{g}_{L_{\frac{5}{3}}(\Omega^t)} + \norm{u}_{L_{10}(\Omega^t)} + \norm{\Theta(t_0)}_{W^{\frac{4}{5}}_{\frac{5}{3}}(\Omega)}\right).
	\end{equation*}
	From \eqref{eq10}, the above inequality and the maximal regularity for the Stokes and parabolic system (see e.g. \cite[Lemmas 6.9 and 6.11]{2012arXiv1205.4046N}) it follows that we can improve the regularity for $(u,\Theta)$ as we need. In particular
	\begin{equation*}
		\begin{aligned}
			\sup_{t \geq 0} \norm{u(t)}^4_{H^1(\Omega)} &\leq c_{\text{data}}, \\
			\sup_{t \geq 0} \norm{\nabla \Theta(t)}^2_{L_3(\Omega)} &\leq c_{\text{data}}.
		\end{aligned}
	\end{equation*}
	In consequence, $\Delta(\nu)$ becomes positive and finite for $\nu$ sufficiently large. This ends the proof.
\end{proof}

\begin{rem}
	Observe, that for $u(t_0) = v(t_0)$ and $\Theta(t_0) = \omega(t_0)$ we obtain
	\begin{equation*}
		\norm{V(t)}^2_{L_2(\Omega)} + \norm{Z(t)}^2_{L_2(\Omega)} \leq \nu_r^2c_{\alpha,I,\Omega}\left(\norm{\omega}_{L_2(\Omega^t)}^2 + \norm{\Rot v}_{L_2(\Omega^t)}^2\right).
	\end{equation*} 
	This estimate implies that as $\nu_r \to 0$ the velocity of the micropolar fluid model converges uniformly with respect to time on $[0,\infty)$ in $L_2$ to the velocity field of standard Navier-Stokes model.
\end{rem}

\bibliographystyle{amsalpha}
\bibliography{bibliography}

\end{document}